\documentclass[11pt]{amsart}
\usepackage{amsmath,amsthm,amsfonts,amssymb,times,latexsym,mathrsfs}

\usepackage{hyperref}

\usepackage[dvipsnames,svgnames,table]{xcolor}

\numberwithin{equation}{section}  

\DeclareMathAlphabet{\curly}{U}{rsfs}{m}{n}  

\theoremstyle{remark}

\theoremstyle{plain}

\newtheorem{lem}{Lemma}[section]
\newtheorem{thm}{Theorem}
\newtheorem{cor}[thm]{Corollary}

\numberwithin{equation}{section}

\newcommand{\ZZ}{{\mathbb Z}}

\newcommand{\NN}{{\mathbb N}}

\newcommand{\cA}{\ensuremath{\mathcal{A}}}

\newcommand{\cE}{\ensuremath{\mathcal{E}}}

\newcommand{\cM}{\ensuremath{\mathcal{M}}}

\newcommand{\sR}{\ensuremath{\mathscr{R}}}

\newcommand{\bV}{\ensuremath{\mathbf{V}}}
\newcommand{\bU}{\ensuremath{\mathbf{U}}}
\newcommand{\bW}{\ensuremath{\mathbf{W}}}
\newcommand{\bX}{\ensuremath{\mathbf{X}}}

\newcommand{\PR}{\mathbb{P}}
\newcommand{\E}{\mathbb{E}}

\makeatletter
\renewcommand{\pmod}[1]{\allowbreak\mkern7mu({\operator@font mod}\,\,#1)}
\makeatother

\newcommand{\be}{\begin{equation}}
\newcommand{\ee}{\end{equation}}

\newcommand{\ssum}[1]{\sum_{\substack{#1}}}  
\newcommand{\sprod}[1]{\prod_{\substack{#1}}}  

\newcommand{\eps}{\ensuremath{\varepsilon}}
\newcommand{\er}{\mathrm{e}}

\renewcommand{\le}{\leqslant}

\renewcommand{\ge}{\geqslant}

\newcommand{\fl}[1]{{\ensuremath{\left\lfloor {#1} \right\rfloor}}}  
\renewcommand{\(}{\left(}
\renewcommand{\)}{\right)}

\newcommand{\pfrac}[2]{\left(\frac{#1}{#2}\right)}  

\newcommand{\ms}{\medskip}

\newcommand{\ssc}[2]{\ensuremath{{#1}_{#2}^{\phantom{2}}}}

\begin{document}

\title{Poisson approximation of prime divisors of shifted primes}
\author{Kevin Ford}

\address{Department of Mathematics, University
  of Illinois at Urbana-Champaign}
\email{ford126@illinois.edu}
\begin{abstract}
We develop an analog for shifted primes of the Kubilius model of prime
factors of integers.
We prove a total variation distance estimate for the difference between the model
and actual prime factors of shifted primes, and apply it to show that
the prime factors of shifted primes in disjoint sets behave like independent
Poisson variables.  As a consequence, we establish a transference principle
between the anatomy of random integers $\le x$ and of random shifted primes $p+a$ with $p\le x$.
\end{abstract}

\date{\today}

\maketitle


{\Large\section{Introduction}}

The prime factorization of a random integer
$n\le x$ is well-approximated, in a probabilistic sense, by a
sequence of independent random variables.  This approximation is known 
as the Kubilius model after the pioneering work of Kubilius
\cite{kubilius}; see also \cite{elliott} for a development of the model.  
In this note we develop an analogous 
model for the prime factorization of a \emph{shifted prime}
$p+a$ with $a$ fixed and $p$ a prime drawn at random from $(|a|,x]$.
We prove a total variation distance estimate for the difference between the model
and actual prime factors of shifted primes, and apply it to show that
the prime factors of shifted primes in disjoint sets behave like independent
Poisson variables.  This in turn is used to establish a general transference principle,
which states roughly that any property of the  prime factors of
a random integer $n\le x$, which is weakly dependent on the smallest and largest prime
factors of $n$, also holds for the prime factors of a random 
shifted prime $p+a$ with $p\le x$.  This excludes properties such as
bounding the size of the largest prime factor of $n$ (a notoriously difficult problem) and the property ``$2|n$'', which must be excluded since it occurs for half of integers $n\le x$, but for $\pi(x)-1$ primes $p\le x$ with $n=p+a$ when $a$ is odd,
and for only one prime $p$ when $a$ is even.

 When $a\in \{-1,1\}$,
the distribution of the prime factors of 
numbers $p+a$, $p$ being prime, plays a central role in investigations of
Euler's totient function and the sum of divisors function
(see, for example, \cite{erdos35, EGPS, EP, F98, F99, FK99, FLP1, FP2, PP}),
the Carmichael $\lambda$-function \cite{EPS, FLP2},
orders and primitive roots modulo primes \cite{FGK, GM, HB-primroots, Ho67}, 
Carmichael numbers \cite{AGP, Harman08},
and in computational problems (e.g. primality testing, factorization,
pseudo-random number generators), see e.g., \cite{AKS},
\cite{AGP}, \cite{PM} and Sections 3.3, 3.5, 4.1, 4.5 and 5.4 of \cite{CP}.

\subsection{Kubilius' model of prime factors of integers}

Take a uniformly random integer $n$ drawn from the interval $[1,x]$.
Each such $n$ has a unique prime factorization
\[
n= \prod_{p\le x} p^{v_p},
\]
where $p$ is prime and the exponents $v_p$ are now random variables.
The probability that $v_p=k$ equals
\[
\frac{1}{\fl{x}} \( \fl{\frac{x}{p^k}} - \fl{\frac{x}{p^{k+1}}} \)  = \frac{1}{p^k} 
- \frac{1}{p^{k+1}} + O\pfrac{1}{x},
\]
the error term being relatively small when $p^k$ is small.
Moreover, the variables $v_p$ are quasi-independent; that is, the 
correlations are small, again provided that the primes are small.
This last fact is a consequence of the small sieve, which is the key ingredient
in establishing the validity of the \emph{Kubilius model of integers}.

The model of Kubilius is a sequence of \emph{idealized} random variables 
which removes the error term above and is much easier to compute with.
For each prime $p$, define the random variable $X_p$ that has domain $\NN_0=\{0,1,2,3,4,\ldots\}$
and such that
\[
\PR (X_p=k) = \frac{1}{p^k} - \frac{1}{p^{k+1}} = \frac{1}{p^k} \(1-\frac{1}{p} \) \qquad (k=0,1,2,\ldots).
\]
Moreover, suppose that the variables $X_p$ are mutually independent.
The principal result, first proved by Kubilius and
later sharpened by others, is that  the random vector
\[
\bX_y = (X_p:p\le y)
\]
has distribution close to that of the random vector
\[
\bV_{x,y} = (v_p: p\le y),
\]
provided that $y=x^{o(1)}$.  It is convenient to 
describe the quality of the approximation using the 
total
variation distance $d_{TV}(X,Y)$ between two random variables
 living on the same discrete space $\Omega$:
 \be\label{dTV-def}
 d_{TV}(X,Y) := \sup_{U\subset \Omega} \big| \PR(X\in A)-\PR(Y\in A) \big|.
\ee
In \cite{ten99}, Tenenbaum gives an
asymptotic for $d_{TV}(\bX_y,\bV_{x,y})$ in terms of a convolution
of the Buchstab and Dickman functions,
as well as a simpler universal upper bound which we state here.

\ms

\textbf{Theorem A}.
(Tenenbaum \cite[Th\'eor\`eme 1.1 and (1.7)]{ten99})
Let $2\le y\le x$.  Then, for every $\eps>0$,
\[
d_{TV}(\bX_y,\bV_{x,y}) \ll_\eps u^{-u} + x^{-1+\eps}, \quad u=\frac{\log x}{\log y}.
\]

\ms

Such a bound is very useful in the study of additive arithmetic functions
(see \cite{elliott} and \cite{kubilius}), where the statistical behavior
of
\[
f(n) = \sum_{p\le x} f(p^{v_p})
\]
is well approximated by the behavior of the model function
\[
\tilde{f} = \sum_{p\le x} f(p^{X_p}).
\]

\subsection{A Kubilius type model for shifted primes}

Fix a nonzero integer $a$. 
It is expected that the distribution of the prime
factors of a random shifted prime $p+a$ with $p\le x$ behaves very much
like the distribution of the prime factors of a random integer in $[1,x]$,
provided that the statistic under consideration depends weakly on the
smallest prime factors of $p+a$; this restriction is needed since if $q$ is prime
and $q\nmid a$, then $q|(p+a)$ for a proportion $\frac{1}{q-1}$ of the primes
by the prime number theorem for arithmetic progressions \cite[\S 20]{Da}. 
This principle has been established for many statistics which are weakly dependent
on both the very smallest prime factors and the very largest prime factors.
The first such result is the seminal 1935 paper \cite{erdos35} of Erd\H os,
who showed that
the number of prime factors of $p+a$ is usually very close to 
$\log\log x$.  This established an analog for shifted primes
of the famous theorem of Hardy and Ramanujan, who
showed in 1917 that most integers $n\le x$ have about $\log\log x$
prime divisors.  To date the strongest results about the
number of prime factors of $p+a$ are due to
Timofeev \cite{Tim95} and the author \cite{HRshift}.
The theory of general
additive functions on shifted primes has been much studied, e.g. Elliott \cite{elliott97}.
However, to the author's knowledge nobody has written down an explicit
comparison analogous to Theorem A. 

We first introduce random variables $W_q$ for prime $q$.  
For $q|a$, we set $W_q=0$ with probability 1,
and for $q\nmid a$ define $W_q$ according to the rule
\be\label{Wq}
\PR(W_q=v) = \frac{1}{\phi(q^v)} - \frac{1}{\phi(q^{v+1})} = \begin{cases}
1-\frac{1}{q-1} & \text{ if } v=0,\\ \frac{1}{q^v} & \text{ if } v\ge 1.
\end{cases}
\ee
Also, the variables $W_q$ are mutually independent.
Thus, by the prime number theorem for arithmetic progressions,
$\PR(W_q=v)$ is the density of primes $p$ for which $q^v \| (p+a)$.
Let $\bW_y$ be the vector of $W_q$ for $q\le y$.
For a randomly drawn prime $p\in (|a|+1,x]$ write
\[
p+a = \prod_{q} q^{u_q}
\]
and let $\bU_{x,y}$ denote the vector $(u_q: q\le y)$.
The lower bound on $p$ ensures that $p$ is odd and $p+a \ge 2$.

Our results depend on the ``level of distribution''
 of primes in progressions with smooth moduli.
 Denote by $P^+(n)$ the largest prime factor of $n$, with the convention that
$P^+(1)=0$.  Let $P^-(n)$ be the smallest prime factor of $n$, with
the convention that $P^-(1)=\infty$.
As is usual, denote by 
$\pi(x)$ the number of primes $p\le x$ and by
$\pi(x;m,b)$ the number of primes $p\le x$
in the progression $b\mod m$.  We need that $\pi(x;m,b)$ is approximately
$\frac{\pi(x)}{\phi(m)}$ uniformly on average over smooth moduli $m$
up to some power of $x$.  
Consider the following hypothesis.

\medskip

\noindent\textbf{Hypothesis $Z(\gamma)$.}  
For any $\eps>0$ there is a $\delta>0$ so that for any $B\ge 1$ and nonzero
integer $a$ we have
\[
 \ssum{m\le x^{\gamma-\eps} \\ P^+(m)\le x^{\delta} \\ (m,a)=1} \bigg|
\pi(x;m,-a) - \frac{\pi(x)}{\phi(m)} \bigg| \ll_{B,\eps,a} \frac{x}{(\log x)^B}.
\]

Hypothesis $Z(\frac12)$ (where we may take $\delta=1$ for any $\eps>0$) is
a consequence of the Bombieri-Vinogradov Theorem
\cite[Chapter 28]{Da},
improving upon an earlier result of Barban \cite{Barban},
who showed Hypothesis $Z(\frac38)$, again with $\delta=1$ for any $\eps>0$.
The best known result is $Z(\frac12+\frac{1}{42})$, due to
Maynard \cite[Corollary 1.2]{M}, which uses ideas originating from the
celebrated work of Zhang \cite{Zhang} and further developed by
the PolyMath8a project \cite{polymath8a}.
We note that the estimates of Zhang, PolyMath8b as well as new work of 
Stadlmann \cite{Sta} establish versions of Hypothesis $Z(\gamma)$
with $\gamma>\frac12$ but with the sum restricted to squarefree $m$.
It is conjectured that $Z(1)$ holds; indeed, this is a simple consequence of
the Elliott-Halberstam conjecture.

\begin{thm}\label{thm:main}
Assume Hypothesis $Z(\gamma)$.
Fix $a\ne 0$, $A>0$ and $0<\alpha < \gamma$.  Then, for $2\le y\le x$ we have
\[
d_{TV}(\bW_y,\bU_{x,y}) \ll_{a,A,\alpha} e^{-\alpha u\log u}  + \frac{1}{(\log x)^A},
\quad u=\frac{\log x}{\log y},
\]
the implied constant in the $\ll-$symbol depending only on $a,A$ and $\alpha$.
\end{thm}

In plain language, Theorem \ref{thm:main} implies that the distribution
of prime factors
of shifted primes which are below $y=x^{o(1)}$ is uniformly well-approximated 
by the vector of random variables $\bW_y$.

The distribution of the large prime factors of shifted primes is
not well understood at present.  It is known that for
infinitely many primes $p$, $P^+(p+a) \le p^{0.2844}$ \cite{Li}
and for infinitely many primes $p$,  $P^+(p+a) \ge p^{0.677}$
\cite{BH98}.  It is conjectured that analogous statements hold with 
$0.2844$ replaced by any positive number, and with $0.677$
replaced by any number $<1$.  We do not address these problems in
this paper.

The additive term $\frac{1}{(\log x)^A}$ appearing in Theorem \ref{thm:main}
arises from the hypothetical
existence of a modulus $q\le (\log x)^{A}$ for which the primes
below $x$ have irregular distribution in progressions modulo $q$, which comes from the
hypothetical existence of a real zero of a Dirichlet $L$-function for a real character 
modulo $q$ which is very close to 1.
Even if $A=A(x)$ is a function tending to $\infty$ arbitrarily slowly,
we cannot at present rule out the existence of such a $q$; see \cite[\S 20--22]{Da} 
for details.
Under suitable assumptions on the real zeros of Dirichlet $L$-functions
attached to real characters,
one can replace the term $1/(\log x)^A$ with a smaller function.
We leave the details to the interested reader.

\subsection{Poisson approximation of prime factors of shifted primes}

For any finite set $T$ of primes, let
\[
\omega(n,T) = \# \{p|n:p\in T\}=\#\{p\in T: v_p>0\},
\qquad \Omega(n,T)=\sum_{p\in T} v_p.
\]
By Theorem \ref{thm:main}, the distribution of  $\omega(p+a,T)$ should be well-approximated by the distribution of the random variable
\be\label{RT}
R_T := \# \{q \in T : W_q \ge 1 \}
\ee
and $\Omega(p+a,T)$ should be 
well-approximated by the distribution of 
\be\label{RTT}
\widetilde{R}_T := \sum_{q\in T} W_q,
\ee
 provided that $\frac{\log y}{\log x}$ is small.

For prime $q$,
\[
\PR(W_q \ge 1) = \frac1{q-1}, \qquad \E W_q = \frac{q}{(q-1)^2}.
\]
Thus we expect that $R_T$ is well-approximated by a Poisson
variable with parameter $H_1(T)$, and $\widetilde{R}_T$ is
well-approximated by a Poisson
variable with parameter $H_1'(T)$, where
\be\label{H1}
H_1(T) := \sum_{q\in T} \frac{1}{q-1}, \qquad
H_1'(T) := \sum_{q\in T} \frac{q}{(q-1)^2}.
\ee
Also define
\[
H_2(T) = \sum_{q\in T} \frac{1}{q^2}.
\]

\begin{thm}\label{thm2}
Assume Hypothesis $Z(\gamma)$.
Fix $a\ne 0$, $A>0$ and $0<\alpha < \gamma$.
Let $2\le y\le x$ and suppose that
 $T_1,\ldots,T_m$ are disjoint nonempty sets of primes in $[2,y]$.
 For each $1\le i\le m$, suppose that either
 $f_i=\omega(p+a,T_i)$ and $\lambda_i = H_1(T_i)$ or that
 $f_i=\Omega(p+a,T_i)$ and $\lambda_i = H_1'(T_i)$.
Assume moreover that 
for each $i$, $Z_i$ is a Poisson random variable with parameter $\lambda_i$ and that
$Z_1,\ldots,Z_m$ are independent.
Then
\[
d_{TV} \Big( (f_1,\ldots,f_m), (Z_1,\ldots,Z_m) \Big) \ll_{a,A,\alpha}  \sum_{j=1}^m \frac{H_2(T_j)}{1+H_1(T_j)} + \er^{-\alpha u\log u}+\frac{1}{(\log x)^A}, \quad u=\frac{\log x}{\log y}.
\]
\end{thm}

The estimate in Theorem \ref{thm2} is uniform in $m$ and in $T_1,\ldots,T_m$.

The analogous statement for the distribution of prime factors
of all integers $n\le x$ has been given by the author \cite{FPoisson}.
Provided that the right side in the conclusion is sufficiently small, this reduces problems about
 the distribution of $(f_1,\ldots,f_m)$ to the analogous problems about the distribution of 
$(Z_1,\ldots,Z_m)$, the latter being much easier especially due to the independence
of $Z_1,\ldots,Z_m$.

As the parameter $\lambda$ of a Poisson random variable approaches $\infty$,
it converges weakly to a normal random variable with mean and variance $\lambda$.
 Also, $H_1(T) \le H_1'(T)$ for any $T$.
Hence, we have the following corollary. 

\begin{cor}
Suppose that $T_1,\ldots,T_m$ are disjoint subsets of primes in $[2,y]$, which may be functions of $x$ (with $m$ a function of $x$ also).

(a) If $u=u(x) \to \infty$ as $x\to \infty$ and $\sum H_2(T_i)/(1+H_1(T_i))\to 0$ as $x\to\infty$, then uniformly over all
subsets $\sR$ of $\NN_0^m$ we have
\[
\Big| \PR( (f_1,\ldots,f_m)\in \sR ) - \PR( (Z_1,\ldots,Z_m)\in \sR ) \Big| = o(1) \quad 
(x\to \infty).
\]

(b) In addition, if $\min_i H_1(T_i)\to \infty$ as $x\to\infty$,
then $(f_1,\ldots,f_m)$ converges weakly to $(N_1,\ldots,N_m)$
as $x\to\infty$,
where $N_1,\ldots,N_m$ are independent Gaussian normal random
variables, and for each $i$, $N_i$ has mean and variance
$\lambda_i$.
\end{cor}

The special case of (b), where $m=1$ and $T_1$ is the set of all primes in $[2,x]$
(not strictly covered by the Corollary, but easily deducible from it) was
first shown by Halberstam \cite{Hal56}.

Finally, we use Theorem \ref{thm2} to prove a general \emph{transference principle},
whereby a property of the prime factors of a random $n\le x$ automatically
holds for the prime factors of a random shifted prime $p+a$, $p\le x$.

\begin{thm}\label{thm:transference}
Assume Hypothesis $Z(\gamma)$.
Fix $a\ne 0$, $A>0$ and $0<\alpha < \gamma$.
Let $2\le y\le x$ and suppose that
 $T_1,\ldots,T_m$ are disjoint nonempty sets of primes in $[2,y]$.
 For each $1\le i\le m$, suppose that either
 $f_i(n)=\omega(n,T_i)$  or that
 $f_i(n)=\Omega(n,T_i)$.  Then, for any subset 
 $\sR \subset \NN_0^m$ we have
 \begin{multline*}
 \Big| \PR \big( (f_1(n),\ldots,f_m(n))\in \sR \big) - 
 \PR \big( (f_1(p+a),\ldots,f_m(p+a))\in \sR \big) \Big|
 \ll_{a,A,\alpha}  \sum_{j=1}^m \frac{H_2(T_j)}{1+\sqrt{H_1(T_j)}} +\\
 + \er^{-\alpha u\log u}+\frac{1}{(\log x)^A}, \quad u=\frac{\log x}{\log y}.
 \end{multline*}
\end{thm}
The first probability is over a random integer $n\le x$ and the second is over
a random prime $p$ satisfying $|a| + 1 <p\le x$.

\begin{cor}
Assume the hypotheses of Theorem \ref{thm:transference}, that $y\le x^{o(1)}$ (as $x\to \infty$)
and either (i) $\min_j H_1(T_j)\to\infty$ as $x\to\infty$ or (ii) $\max_j H_2(T_j)=o(1)$
as $x\to\infty$.  Then, uniformly over all  $\sR \subset \NN_0^m$ we have
\[
\Big|  \PR \big( (f_1(n),\ldots,f_m(n))\in \sR \big) - 
  \PR \big( (f_1(p+a),\ldots,f_m(p+a))\in \sR \big) \Big| = o(1) \qquad (x\to\infty).
\]
\end{cor}
\begin{proof}
In case (i) this follows from Theorem \ref{thm:transference}, since $\sum_j H_2(T_j)=O(1)$
uniformly for any  $T_1,\ldots,T_m$.  In case (ii), if $t=\min_j \min T_j$ then
$t\to\infty$ as $x\to \infty$.  Then $\sum_j H_2(T_j) \ll \sum_{p\ge t}1/p^2=o(1)$ as $x\to\infty$
and the claim follows from Theorem \ref{thm:transference}.
\end{proof}


\subsection{An application}

To illustrate the use of the transference principle, Theorem \ref{thm:transference},
we derive the law of the iterated logarithm for the random function $\omega(p+a,[2,t])$
of $t$ from the analogous result for the random function $\omega(n,[2,t])$.  For $t$
with $\log_4 t\ge 1$ and positive integer $n$, define
\[
\Lambda(n,t) := \frac{\omega(n,[2,t])-\log_2 t}{\sqrt{2\log_2 t \cdot \log_4 t}},
\]
where $\log_k$ denotes the $k$-th iterate of the logarithm.
Hall and Tenenbaum \cite[Theorem 11]{HT} showed that for any $\eps>0$
and any increasing function $\xi(x)$ satisfying $\xi(x)\to \infty$
as $x\to\infty$ and $\xi(x)\le \log x$ for all $x$, we have
\begin{equation}\label{eq:LIL-n}
\inf_{\xi(x)<t\le x} \Lambda(n,t) \in [-1-\eps,-1+\eps], \qquad 
\sup_{\xi(x)<t\le x} \Lambda(n,t) \in [1-\eps,1+\eps] 
\end{equation}
for all but $o(x)$ integers up to $x$, as $x\to\infty$.

\begin{thm}\label{thm:LIL}
Fix $a\ne 0$ and $\eps>0$.
Consider any increasing function $\xi(x)$ satisfying $\xi(x)\to \infty$
as $x\to\infty$ and $\xi(x)\le \log x$ for all $x$.  Then
\[
\inf_{\xi(x)<t\le x} \Lambda(p+a,t) \in [-1-\eps,-1+\eps], \qquad 
\sup_{\xi(x)<t\le x} \Lambda(p+a,t) \in [1-\eps,1+\eps] 
\]
for all but $o(x/\log x)$ primes $p$ in $(|a|+1,x]$, as $x\to\infty$.
\end{thm}

\begin{proof}
Let $x$ be large and set $y=x^{1/\log_3 x}$.  All integers $\le x$ have at most
$\log_3 x$ prime factors larger than $y$, and it follows that
\begin{equation}\label{eq:Lambda-x1-x}
\Lambda(n,t) - \Lambda(n,y) \ll \frac{\log_3 x}{\sqrt{\log_2 x\cdot \log_4 x}} \qquad (y\le t\le x, \text{all } n\le x).
\end{equation}
Hence, by \eqref{eq:LIL-n},  for any $\eps>0$ we have
\begin{equation}\label{eq:LIL-n-y}
\inf_{\xi(x)<t\le y} \Lambda(n,t) \in [-1-\eps/2,-1+\eps/2], \qquad 
\sup_{\xi(x)<t\le y} \Lambda(n,t) \in [1-\eps/2,1+\eps/2]
\end{equation}
for all but $o(x)$ integers up to $x$, as $x\to\infty$.
Denote by $q_1,\ldots,q_k$ the primes in $(\xi(x),y]$, let
$T_1$ be the set of primes in $[2,\xi(x)]$, let $T_j = \{q_{j-1}\}$ for $2\le j\le k+1$
and set $f_i=\omega(n,T_i)$ for each $i$.
The condition \eqref{eq:LIL-n-y} is then equivalent to 
$(\omega(n,T_1),\ldots,\omega(n,T_{k+1}))\in \sR$ for some set $\sR \subseteq \NN_0^{k+1}$.  Apply Theorem \ref{thm:transference} with $\alpha=\frac12$ (allowed since Hypothesis $Z(\gamma)$ holds for some $\gamma>\frac12$), $A=1$
and $u=\log_3 x$, and we see that
\[
\sum_{j=1}^{k+1} \frac{H_2(T_j)}{1+\sqrt{H_1(T_j)}} 
 + \er^{-\alpha u\log u}+\frac{1}{\log x} \ll \frac{1}{\log_2 x}+\frac{1}{\sqrt{H_1(T_1)}}+\sum_{q>\xi(x)}\frac{1}{q^2} \ll \frac{1}{\sqrt{\log_2 \xi(x)}}.
\]
Therefore, \eqref{eq:LIL-n-y} holds with $n=p+a$
for all but $o(x/\log x)$ primes $p\in (|a|+1,x]$, as $x\to\infty$.
Recalling \eqref{eq:Lambda-x1-x}, this completes the proof.
\end{proof}

A variant of Theorem \ref{thm:LIL} is needed in recent work 
on primes with small primitive roots \cite{FGGK}.

\medskip

\textbf{Acknowledgements.}  The author thanks Steve Fan, Mikhail Gabdullin, G\'erald Tenenbaum
and the anonymous referees
for helpful comments on earlier versions of the paper.  The author was supported in part by
National Science foundation grants DMS-1802139 and DMS-2301264.

\medskip

%
{\Large \section{Total variation distance inequalities}}

The following are standard; for completeness we give full proofs.

\begin{lem}\label{dTV-eq}
For random variables $X,Y$ defined on the same countable space $\Omega$, we have
\[
d_{TV}(X,Y) = \frac12 \sum_{\omega\in \Omega} \big| \PR(X=\omega)-\PR(Y=\omega) \big|
\]
\end{lem}

\begin{proof}
The supremum in \eqref{dTV-def} occurs when $U=U^+:=\{\omega\in \Omega: \PR(X=\omega)>\PR(Y=\omega)\}$ and when $U=U^-:=\{\omega\in \Omega: \PR(X=\omega)<\PR(Y=\omega)\}$, thus
\begin{align*}
2d_{TV}(X,Y) &= \PR(X\in U^+) - \PR(Y\in U^+) + \PR(Y\in U^-) - \PR(X\in U^-)\\
&= \sum_{\omega\in \Omega} \big| \PR(X=\omega)-\PR(Y=\omega) \big|.\qedhere
\end{align*}
\end{proof}

\begin{lem}\label{dTV-sum}
If $X_1,\ldots,X_m$ are independent discrete random
variables,
 and $Y_1,\ldots,Y_m$ are independent discrete random variables
(with $Y_j$ having the same domain as $X_j$ for each $j$), then
\[
d_{TV} \big( (X_1,\ldots,X_m),(Y_1,\ldots,Y_m) \big) \le \sum_{j=1}^m d_{TV} (X_j,Y_j).
\]
\end{lem}

\begin{proof}
We begin with the following identity
\[
a_1 \cdots a_m - b_1\cdots b_m = \sum_{j=1}^m  (a_j-b_j)\prod_{i<j}
a_i \prod_{i>j} b_i,
\]
valid for all real numbers $a_1,b_1,\ldots,a_m,b_m$.
Denoting by $\Omega$ the domain of $(X_1,\ldots,X_m)$, and writing
$a_i=\PR(X_i=\omega_i)$, $b_i=\PR(Y_i=\omega_i)$,
 we then have
\begin{align*}
d_{TV}  ( (X_1,\ldots,X_m),(Y_1,\ldots,Y_m) ) &= \frac12 \sum_{
  (\omega_1,\ldots,\omega_m)\in \Omega} \big| \PR(X_j=\omega_j, 1\le
j\le m) - \PR(Y_j=\omega_j, 1\le j\le m) \big| \\
 &= \frac12 \sum_{  (\omega_1,\ldots,\omega_m)\in \Omega} |a_1\cdots
a_m - b_1\cdots b_m|\\
 &\le \frac12 \sum_{j=1}^m \sum_{\omega_j} |a_j-b_j| \sum_{\omega_i \;
  (i\ne j)}  \prod_{i<j} a_i \prod_{i>j} b_i \\
&= \frac12 \sum_{j=1}^m \sum_{\omega_j} |a_j-b_j| \\
&=  \sum_{j=1}^m d_{TV} (X_j,Y_j).\qedhere
\end{align*}
\end{proof}

\begin{lem}\label{dTV-Poisson}
Let $X$ be a Poisson random variable with parameter $\lambda$ and $Y$ be a Poisson random variable with parameter $\lambda'$, where $0<\lambda<\lambda'$.  Then
\[
d_{TV}(X,Y) \ll \frac{\lambda'-\lambda}{1+\sqrt{\lambda}}.
\]
\end{lem}

\begin{proof}
First suppose that $\lambda'<1$.  Then $\er^{-\lambda'}(\lambda')^k \ge \er^{-\lambda}\lambda^k$ for all $k\ge 1$.  Hence, by Lemma \ref{dTV-eq},
\[
2 d_{TV}(X,Y) = \er^{-\lambda}-\er^{-\lambda'}+\sum_{k=1}^\infty \frac{\er^{-\lambda'}(\lambda')^k - \er^{-\lambda}\lambda^k}{k!} = 2\big( \er^{-\lambda}-\er^{-\lambda'}\big)
\le 2(\lambda'-\lambda).
\]
Now assume $\lambda' \ge 1$.  The conclusion is trivial if $\lambda' \ge  \lambda+\sqrt{\lambda}$
thus we may assume that $1 \le \lambda' \le \lambda+\sqrt{\lambda}$.
By Pinsker's inequality \cite[Theorem 2.16]{Massart},
\[
d_{TV}(X,Y) \le \sqrt{(1/2)D_{KL}(X \| Y)}, 
\]
where $D_{KL}(X \| Y)$ is the Kullback-Leibler divergence (also called relative entropy)
between $X$ and $Y$, given by
\[
D_{KL}(X \| Y) := \sum_{k=0}^\infty \PR(X=k) \log \frac{\PR(X=k)}{\PR(Y=k)}.
\]
An easy calculation gives
\begin{align*}
D_{KL}(X \| Y) &= (\lambda'-\lambda) - \lambda \log (\lambda'/\lambda)
\ll \frac{(\lambda'-\lambda)^2}{\lambda},
\end{align*}
and the lemma follows.
\end{proof}

\medskip

%
%
{\Large \section{Sieving by small primes}}
%
%

Define the set
\[
\cE(x,y;\theta,A) = \bigg\{ m \le x^{\theta}, P^+(m)\le y : \ssum{d\le x^\theta/m\\ P^+(d)\le y \\ (dm,a)=1} \bigg|
\pi(x;md;-a) - \frac{\pi(x)}{\phi(md)} \bigg| \ge \frac{\pi(x)}{\phi(m) (\log x)^{A} } 
\bigg\}.
\]

\medskip

\begin{lem}\label{lem:Eset}
Assume Hypothesis $Z(\gamma)$.
For any $A>0$, $\theta$ satisfying 
\be\label{theta}
\frac12 < \theta < \gamma,
\ee
$\delta$ sufficiently small (as a function of $\theta$ only), and 
$y$ satisfying
\be\label{y}
 y \le x^{\delta}
\ee
we have
\[
\sum_{m\in \cE(x,y;\theta,A)} \frac{1}{\phi(m)} \ll_{A,\theta,a}\, \frac{1}{(\log x)^{A+1}}.
\]
\end{lem}

\begin{proof}
  We have
\begin{align*}
\frac{\pi(x)}{(\log x)^{A}} \sum_{m\in \cE(x,y;\theta,A)} \frac{1}{\phi(m)} &\le 
\ssum{m\le x^{\theta} \\ P^+(m)\le y} \ssum{d\le x^{\theta}/m \\ P^+(d)\le y \\ (dm,a)=1} 
\bigg| \pi(x;dm,-a)-\frac{\pi(x)}{\phi(dm)} \bigg| \\
&\le \ssum{s\le x^{\theta} \\ P^+(s)\le y \\ (s,a)=1} \tau(s) 
\bigg| \pi(x;s,-a)-\frac{\pi(x)}{\phi(s)} \bigg|\\
&\ll x^{1/2} \ssum{s\le x^{\theta} \\ P^+(s)\le y \\ (s,a)=1} \frac{\tau(s)}{s^{1/2}} 
\bigg| \pi(x;s,-a)-\frac{\pi(x)}{\phi(s)} \bigg|^{1/2}.
\end{align*}
Let $B=4A+10$.
By a standard application of Cauchy-Schwartz, the right side is
\begin{align*}
\le  x^{1/2} \bigg( \sum_{s\le x} \frac{\tau^2(s)}{s} \bigg)^{1/2}
\Bigg( \ssum{s\le x^{\theta} \\ P^+(s)\le y \\ (s,a)=1}  
\bigg| \pi(x;s,-a)-\frac{\pi(x)}{\phi(s)} \bigg| \Bigg)^{1/2} \ll
\frac{x}{(\log x)^{2A+2}},
\end{align*}
using Hypothesis $Z(\gamma)$ in the last step.
\end{proof}

\begin{lem}\label{lem:gym}
For positive integer $m$ with $P^+(m)\le y$ and factorization $m=\prod_{q\le y} q^{w_q}$, define
\[
g_y(m) = \PR \big( W_q = w_q\; \forall q\le y  \big).
\] 
Then $g_y(m)=0$ if $(a,m)>1$ or if $a\equiv m\equiv 1\pmod{2}$, and otherwise equals
\be\label{gym}
g_y(m) = \frac{1}{\phi(m)}\sprod{q\le y \\ q\nmid a} \(1 - \frac{\phi(m)}{\phi(qm)}\) = \frac{1}{m} \sprod{q\le y \\ q\nmid am} \(1-\frac{1}{q-1}\).
\ee
\end{lem}

\begin{proof}
This follows immediately from \eqref{Wq}.
\end{proof}

For $m\in \NN$ define
\[
\Phi_m(x,y) = \# \bigg\{|a|+1 < p\le x : p\equiv -a\pmod{m}, P^-\pfrac{p+a}{m} > y \bigg\}.
\]
We estimate $\Phi_m(x,y)$ with sieve methods, starting with the base set
\[
\cA = \Big\{ \frac{p+a}{m} : |a|+1<p\le x, p\equiv -a\pmod{m} \Big\}.
\]
Since $\# \{b\in \cA : q|b \} = \pi(x;qm,-a) + O_a(1)$ for primes $q$,
heuristically we expect that $\Phi_m(x,y)$ is approximately $\pi(x) g_y(m)$.

We note that Mertens' theorem and \eqref{gym} imply that
\be\label{gy-lower}
\frac{1}{\phi(m)\log y} \ll_a g_y(m) \ll_a \frac{1}{\phi(m)\log y}
\qquad \big( (a,m)=1,2|am,P^+(m)\le y \big).
\ee

\medskip

\begin{lem}\label{lem:Phi}
Fix $A>0$ large and $\eps>0$ small.  Assume Hypothesis $Z(\gamma)$, $y\ge 2$, \eqref{theta} and \eqref{y}.
 Uniformly for $m\le x^{\theta}$, with $m\not\in \cE(x,y;\theta,A+1)$ and
 $P^+(m)\le y$,
\[
\Phi_m(x,y) = \pi(x) g_y(m) \( 1 + O_{a,\eps} \( e^{-(1-\eps) \ssc{u}{m}\log (\ssc{u}{m}+1)} + 
\frac{1}{(\log x)^A} \)\),
\]
where we define
\[
\ssc{u}{m}=\frac{\log (x^\theta/m)}{\log y}.
\]
\end{lem}

\begin{proof}
If $g_y(m)=0$ then either $(a,m)>1$ or $a\equiv m\equiv 1\pmod{2}$, and in this
case $\Phi_m(x,y)=0$ since $\Phi_m(x,y)$ counts only primes $p>|a|+1$.  The lemma
holds in this case.
Also if $y>x^\theta/m$ then $u_m\le 1$ and, in light of \eqref{gy-lower}, the lemma is claiming only
that $\Phi_m(x,y) \ll \pi(x)/(\phi(m) \log y)$.  This bound follows from an upper bound sieve
without the need for any prime number estimates.

Now assume $y\le x^{\theta}/m := D$.
When $g_y(m)\ne 0$, the lemma follows quickly from the  ``Fundamental Lemma'' of
the sieve, e.g., Theorem 2.5' in \cite{HR}.
We have
\be\label{Phim-sieve}
\Phi_m(x,y) = \pi(x) g_y(m) \Big( 1 + O_\eps(e^{-(1-\eps) \ssc{u}{m}\log \ssc{u}{m}}) \Big) +
O_a\bigg( D+ \ssum{d\le D \\ P^+(d)\le y \\ (dm,a)=1} \Big|\pi(x;dm,-a)-\frac{\pi(x)}{\phi(dm)}\Big| \bigg).
\ee
Since $m\not \in \cE(x,y;\theta,A+1)$, \eqref{gy-lower} implies that
\[
\ssum{d\le D \\ P^+(d)\le y \\ (dm,a)=1} \Big|\pi(x;dm,-a)-\frac{\pi(x)}{\phi(dm)}\Big| \le \frac{\pi(x)}{\phi(m) (\log x)^{A+1}} \ll_a \frac{\pi(x)g_y(m)}{(\log x)^A}.
\]
Finally, by \eqref{y}, $D\ll \pi(x) g_y(m) (\log x)^{-A}$ as well.
\end{proof}

\begin{lem}\label{Phi-crude}
For any $2\le y\le x$ and $m \le \frac{x+a}{y}$,
\[ 
\Phi_m(x,y) \ll_a \frac{x}{\phi(m) \log^2 y}.
\]
\end{lem}

\begin{proof}
This follows quickly from the upper-bound sieve, e.g. Theorem 2.2 in \cite{HR}.

\end{proof}

Our last result of this section is an easy 
estimate over smooth numbers.

\begin{lem}\label{smooth}
If $y\ge 2$, $t\ge 2$ and $y\ge (\log t)^2$ then 
\[
\ssum{m \ge t \\ P^+(m)\le y} \frac{1}{\phi(m)} \ll (\log y) e^{-s\log (s+1)},
\quad s=\frac{\log t}{\log y}.
\]
\end{lem}

\begin{proof}
 Let  $c>0$ be sufficiently large.
  The sum is trivially $O(\log y)$, and this suffices for $s\le c$.
Now assume $s>c$, so that $y<t^{1/c}$ and also $t\ge 2^{c}$; that is, we may assume
that $t$ is sufficiently large.
For $(\log t)^2 \le y \le t^{1/c}$, write
\[
\frac{1}{\phi(m)} = \sum_{m=dk} \frac{\mu^2(d)}{d\phi(d) k}
\]
and sum over $m$.  This yields
\begin{equation}\label{sum-d}
\ssum{m \ge t \\ P^+(m)\le y} \frac{1}{\phi(m)} = \sum_{P^+(d)\le y} \frac{\mu^2(d)}{d\phi(d)} \ssum{k \ge t/d \\ P^+(k)\le y} \frac{1}{k}.
\end{equation}
The inner sum is $O(\log y)$ and therefore the contribution from
$d>t^{2/3}$ is $O(t^{-2/3}\log y)$.
Now assume that $1\le d\le t^{2/3}$.
Fix $\eps=\frac{1}{10}$.
By \cite[Ch. III.5, Exercise 293(c)]{Tenbook}, 
\[
\Psi(x,y) := \# \{n\le x: P^+(n)\le y \} \ll x u^{-u}+x^{\eps}, \qquad u=\frac{\log x}{\log y},
\]
uniformly for all $x\ge 2$ and $y\ge 2$.  Let
\[
v=\frac{\log \frac{t}{d}}{\log y} = s - \frac{\log d}{\log y},
\]
so that $v \ge s/3 \ge 10$.
 By partial summation,
\begin{align*}
\ssum{k\ge t/d \\ P^+(k)\le y} \frac{1}{k}  \le \int_{t/d}^\infty \frac{\Psi(w,y)}{w^2}\, dw
\ll_\eps \int_{t/d}^\infty w^{-1} e^{-\frac{\log w}{\log y}{\log v}}+\frac{1}{w^{2-\eps}}\, dw
\ll (\log y) e^{-v\log v} + (d/t)^{1-\eps}.
\end{align*}

The lower bound $y\ge (\log t)^2$ implies that
$\frac{1+\log s}{\log y} \le \frac{1}{2}$ and thus by the mean
value theorem,
\[
v\log v \ge s\log s - (1+\log s)\frac{\log d}{\log y} \ge
s\log s - \frac{\log d}{2}.
\]
It follows that the sum of terms on the right side of \eqref{sum-d} corresponding to 
$d\le t^{2/3}$ is $\ll (\log y) e^{-s\log s} + t^{\eps-1}$.
Since $s\log s \le \frac12 \log t$ and $s\log(s+1)=s\log s + O(1)$,
the right side of \eqref{sum-d} is 
\[
\ll  (\log y)e^{-s\log(s+1)} + (\log y)t^{-2/3} \ll (\log y)e^{-s\log(s+1)},
\]
 as desired.
\end{proof}

{\Large \section{The proof of Theorem \ref{thm:main}}}

By taking the implied constant in the conclusion of Theorem \ref{thm:main}
sufficiently large, we may assume without loss of generality that
$y$ satisfies \eqref{y}. Furthermore, if $y\le x^{1/\log_2 x}$, that is, $u\ge \log_2 x$,
then $e^{-\alpha u \log u} \ll_{a,A} (\log x)^{-A}$.  Hence, the conclusion in 
Theorem \ref{thm:main} for $y<x^{1/\log_2 x}$ follows from the conclusion for
$y=x^{1/\log_2 x}$.  We therefore may assume that
\begin{equation}\label{y-range-2}
x^{1/\log_2 x} \le y \le x^{\delta}.
\end{equation}
 We also assume that $x$ is sufficiently large
in terms of $a$.
Fix $\theta$ with $\alpha < \theta < \gamma$, $\eps>0$ satisfying 
$\theta-\eps > \alpha$, and $\delta>0$ sufficiently small as a function on $\theta$.
Recall that
\[
u : = \frac{\log x}{\log y}.
\]

Using Lemmas \ref{dTV-eq} and \ref{lem:gym}, we see that
\be\label{dTV-Phim}
d_{TV}(\bW_y, \bU_{x,y}) = \frac12 \sum_{P^+(m)\le y} 
\bigg| \frac{\Phi_m(x,y)}{\pi^*(x)}-g_y(m) \bigg|,
\ee
where $\pi^*(x) = \pi(x) - \pi(|a|+1)$.
As noted earlier, when $g_y(m) = 0$, we have $\Phi_m(x,y)=0$ and thus
such $m$ do not contribute to \eqref{dTV-Phim}.
Now break up the set $\{m\in \ZZ : P^+(m)\le y, g_y(m)\ne 0 \}$ into five pieces
(the conditions $P^+(m)\le y$ and $g_y(m)\ne 0$ assumed in all):
\begin{align*}
\cM_1 &= \{ m < x^\theta, m\not\in \cE(x,y;\theta,A+1) \}, \\
\cM_2 &= \{ m < x^{\theta}, m\in \cE(x,y;\theta,A+1) \}, \\
\cM_3 &= \{ x^{\theta} \le m \le (x+a)/y \}, \\
\cM_4 &= \{ (x+a)/y < m\le x+a \}, \\
\cM_5 &= \{ m>x+a \}.
\end{align*}

Throughout our proof, implied constants may depend on $a$, $A$, $\eps$ and $\theta$ only.  Now 
\[
\frac{1}{\pi^*(x)}=\frac{1}{\pi(x)} + O_a\pfrac{\log^2 x}{x^2}.
\]
Using Lemma \ref{lem:Phi} and \eqref{gy-lower}, 
\begin{align*}
\sum_{m\in \cM_1} \bigg| \frac{\Phi_m(x,y)}{\pi^*(x)}-g_y(m) \bigg| &\ll x^{-0.9}+
 (\log x)^{-A} \sum_{P^+(m)\le y} g_y(m) +  \ssum{m < x^{\theta} \\ P^+(m)\le y} g_y(m) e^{-(1-\eps)\ssc{u}{m}\log \ssc{u}{m}} \\
&\ll (\log x)^{-A} + \frac{1}{\log y}  \ssum{m< x^{\theta} \\ P^+(m)\le y} \frac{1}{\phi(m)} e^{-(1-\eps)\ssc{u}{m}\log \ssc{u}{m}}.
\end{align*}
We break the summation over $m$ on the right side into intervals $I_\ell=[y^{\ell},y^{\ell+1})$, where $0\le \ell \le \theta u$.
For $m\in I_\ell$, 
\[
\ssc{u}{m}\log (\ssc{u}{m}+1) = (\theta u-\ell)\log(\theta u - \ell+1) + O(\log u).
\]
Using Lemma \ref{smooth} when $\ell\ge 1$ (and recalling \eqref{y-range-2}), and Mertens' theorem when $\ell=0$, we obtain
\begin{align*}
\frac{1}{\log y} \ssum{m< x^{\theta} \\ P^+(m)\le y} \frac{1}{\phi(m)} e^{-(1-\eps)\ssc{u}{m}\log \ssc{u}{m}}
&\ll u^{O(1)}\sum_{0\le \ell \le \theta u} e^{-(1-\eps)B_\ell}, \\
&\qquad \text{where }B_\ell =  (\theta u-\ell)\log(\theta u - \ell+1) + \ell \log (\ell+1).
\end{align*}
Since the function $f(w)=w\log(w+1)$ is convex for $w\ge 0$,
\[
B_\ell \ge 2(\theta u/2)\log(\theta u/2+1) = \theta u \log u - O(u).
\]
Thus,
\be\label{M1}
\sum_{m\in \cM_1} \bigg| \frac{\Phi_m(x,y)}{\pi^*(x)}-g_y(m) \bigg| \ll 
(\log x)^{-A} + u^{O(1)} e^{-(1-\eps)\theta u\log u+O(u)} \ll (\log x)^{-A}+e^{-\alpha u\log u}.
\ee

When $m\in \cM_2$ we use the crude bound $\Phi_m(x,y) \ll x/m$ and thus,
by Lemma \ref{lem:Eset},
\be\label{M2}
\sum_{m\in \cM_2} \bigg| \frac{\Phi_m(x,y)}{\pi^*(x)}-g_y(m) \bigg| \ll
(\log x) \sum_{m\in \cE(x,y;\theta,A+1)} \frac{1}{\phi(m)} \ll (\log x)^{-A}.
\ee

When $m\in \cM_3$, we combine Lemma \ref{Phi-crude}, Lemma \ref{smooth}  and \eqref{gy-lower}, obtaining
\be\label{M3}
\sum_{m\in \cM_3} \bigg| \frac{\Phi_m(x,y)}{\pi^*(x)}-g_y(m) \bigg| \ll
\frac{u}{\log y} \ssum{P^+(m) \le y \\ m \ge x^\theta} \frac{1}{\phi(m)}
\ll u e^{-\theta u\log (\theta u+1)} \ll e^{-\alpha u\log u}.
\ee

When $m>(x+a)/y$, $\Phi_m(x,y)$ is 1 if $m-a$ is prime larger than $|a|+1$ and $m-a\le x$, and zero otherwise.
  Start with the triangle inequality and \eqref{gy-lower},
  and use the same analysis as in the case $m\in \cM_3$:
\begin{align*}
\sum_{m\in \cM_4} \bigg| \frac{\Phi_m(x,y)}{\pi^*(x)}-g_y(m) \bigg| &\ll 
\frac{\log x}{x} \sum_{m\in \cM_4} \Phi_m(x,y) + \frac{1}{\log y}
\sum_{m\in \cM_4} \frac{1}{\phi(m)}\\
&\ll \frac{\log x}{x} \# \{|a|+1< p\le x : P^+(p+a)\le y \} + e^{-(1-\eps)u\log u},
\end{align*}
where we applied \eqref{y} in the last step, assuming that $\delta$ is 
sufficiently small.
By Theorem 1 of \cite{PS}\footnote{This is stated only for $a=-1$ but the proof works for any $a$ with no essential changes.} and standard bounds for the Dickman function
$\rho(u)$ (see, e.g., \cite[Ch. III.5]{Tenbook}),
\[
\# \{|a|+1 <p\le x : P^+(p+a)\le y \} \ll \pi(x) u \rho(u)
\ll \pi(x) e^{-(1-\eps)u\log u}.
\]
Thus,
\be\label{M4}
\sum_{m\in \cM_4} \bigg| \frac{\Phi_m(x,y)}{\pi^*(x)}-g_y(m) \bigg| \ll
e^{-(1-\eps)u\log u}. 
\ee

Finally, when $m\in \cM_5$, $\Phi_m(x,y)=0$ and thus by \eqref{gy-lower}
and Lemma \ref{smooth} we have
\be\label{M5}
\sum_{m\in \cM_5} \bigg| \frac{\Phi_m(x,y)}{\pi^*(x)}-g_y(m) \bigg| \ll
\frac{1}{\log y}\ssum{m>x+a \\ P^+(m)\le y} \frac{1}{\phi(m)} \ll
e^{-(1-\eps)u\log u}.
\ee

Gathering the estimates \eqref{M1}--\eqref{M5}, and recalling \eqref{dTV-Phim},
the proof is complete.

%

\bigskip

{\Large \section{Poisson approximation: the proof of Theorem \ref{thm2}}}

Let $z$ be a complex number with $|z| \le 1.9$.
From the definitions \eqref{RT} and \eqref{RTT} we have
\begin{align*}
\E z^{R_T} &= \prod_{q\in T} \(1 + \frac{z-1}{q-1}\) = 
e^{(z-1)H_1(T)}\(1+ O \(|z-1|^2 H_2(T)\) \)
\end{align*}
and
\begin{align*}
\E z^{\widetilde{R}_T} = \prod_{q\in T} \(1+ \frac{q(z-1)}{(q-1)(q-z)}\)
=e^{(z-1)H'_1(T)}\(1+ O \(|z-1|^2 H_2(T)\) \).
\end{align*}
These are analogs of (3.2) and (3.3) in \cite{FPoisson}.
The proofs of \cite[Theorems 5,6]{FPoisson}  then imply the following.

\bigskip

\begin{lem}\label{dTV-cor}
Let $T$ be a finite  subset of primes.  If
$Z$ is a Poisson variable with parameter $H_1(T)$ then
\[
d_{TV}(R_T,Z) \ll \frac{H_2(T)}{1+H_1(T)}
\]
and if $Z'$ is a Poisson random variable with parameter $H_1'(T)$ then
\[
d_{TV}(\widetilde{R}_T,Z') \ll \frac{H_2(T)}{1+H_1(T)},
\]
 \end{lem}
 
Now suppose that  $T_1,\ldots,T_m$ are disjoint nonempty sets of primes in $[2,y]$.
 For each $1\le i\le m$, suppose that either
 (i) $f_i=\omega(p+a,T_i)$ and $R^{(i)}=R_T$  or that (ii)
$f_i=\Omega(p+a,T_i)$ and $R^{(i)}=\widetilde{R}_T$.
Let $Z_1,\ldots,Z_m$ be as in Theorem \ref{thm2}.
We now combine Lemma \ref{dTV-cor} with Lemma \ref{dTV-sum} to obtain
\[
d_{TV}\big( (R^{(1)},\ldots,R^{(m)}), (Z_1,\ldots,Z_m) \big) \ll
\sum_{i=1}^m \frac{H_2(T_i)}{1+H_1(T_i)}.
\]
Theorem \ref{thm2} now follows from Theorem \ref{thm:main}
and the triangle inequality for $d_{TV}$.

{\Large \section{The transference principle: proof of Theorem \ref{thm:transference}}}

Assume Hypothesis $Z(\gamma)$ with $\gamma>0$ 
and that $0 < \alpha < \gamma$.
Define $F(n) = (f_1(n),\ldots,f_m(n))$. Let
\[
H(T) = \sum_{i\in T} \frac{1}{i}.
\]
Notice that $H(T)$, $H_1(T)$ and $H_1'(T)$ are all within $H_2(T)$ of each other.
If $f_i(n)=\omega(n;T_i)$ let $Z_i'$ be a Poisson random variable with parameter $H(T_i)$,
and if  $f_i(n)=\Omega(n;T_i)$ let $Z_i'$ be a Poisson random variable with parameter $H_1(T_i)$.
Assume also that $Z_1',\ldots,Z_m'$ are independent.
Then, by Theorem 1 of \cite{FPoisson},
\[
\big| \PR(F(n)\in \sR) - \PR( (Z_1',\ldots,Z_m')\in \sR ) \big|
\ll \sum_{j=1}^m \frac{H_2(T_j)}{1+H_1(T_j)} +  e^{-u\log u}, \quad u= \frac{\log x}{\log y}.
\]
On the other hand, Theorem \ref{thm2} implies
\[
\big| \PR(F(p+a)\in \sR) - \PR( (Z_1,\ldots,Z_m)\in \sR ) \big|
\ll_{a,A,\alpha} \sum_{j=1}^m \frac{H_2(T_j)}{1+H_1(T_j)} + 
e^{-\alpha u\log u} + \frac{1}{(\log x)^A}.
\]
Applying Lemmas \ref{dTV-sum} and \ref{dTV-Poisson}, we have
\begin{align*}
\big| \PR( (Z_1',\ldots,Z_m')\in \sR ) - \PR( (Z_1,\ldots,Z_m)\in \sR ) \big|
&\le d_{TV} \big( (Z_1',\ldots,Z_m'), (Z_1,\ldots,Z_m) \big) \\
&\le \sum_{j=1}^m d_{TV} (Z_j', Z_j) \\
&\ll \sum_{j=1}^m \frac{H_2(T_i)}{1+\sqrt{H_1(T_i)}}.
\end{align*}
The theorem now follows from the triangle inequality.

\bigskip


\end{document}